\documentclass[12pt]{amsart}
\usepackage{amsmath} 
\usepackage{amsfonts} 
\usepackage{amssymb} 
\usepackage{amsthm} 
\usepackage{amscd} 
\usepackage{enumerate}
\usepackage{graphicx,xcolor}
\usepackage{amsopn}
\numberwithin{equation}{section}
\usepackage{wasysym} 
\usepackage{yhmath} 

\theoremstyle{definition}

\newtheorem{definition}{Definition}[section] 
\newtheorem*{definition*}{Definition} 

\theoremstyle{plain}

\newtheorem{lemma}{Lemma}[section]
\newtheorem*{lemma*}{Lemma} 

\newtheorem{proposition}{Proposition}[section]
\newtheorem*{proposition*}{Proposition} 

\newtheorem*{theorem*}{Theorem} 
\newtheorem*{maintheorem}{Main Theorem} 

\newtheorem*{corollary*}{Corollary} 

\newtheorem*{consequence*}{Consequence} 

\newtheorem*{conjecture*}{Conjecture} 

\theoremstyle{definition}

\newtheorem{remark}{Remark}[section]
\newtheorem*{remark*}{Remark} 

\newtheorem{example}{Example}[section]
\newtheorem*{example*}{Example} 

\newtheorem*{question*}{Question} 

\newtheorem*{exercise*}{Exercise} 

\theoremstyle{definition}

\newtheorem*{déf*}{Définition} 

\theoremstyle{plain}

\newtheorem*{lem*}{Lemme} 

\newtheorem*{prop*}{Proposition} 

\newtheorem*{thm*}{Théorème} 

\newtheorem*{cor*}{Corollaire} 

\newtheorem*{csq*}{Conséquence} 

\newtheorem*{conj*}{Conjecture} 

\theoremstyle{definition}

\newtheorem*{rem*}{Remarque} 

\newtheorem*{ex*}{Exemple} 

\newtheorem*{que*}{Question} 

\newtheorem*{exo*}{Exercice} 

\theoremstyle{remark}

\newcommand{\disp}{\displaystyle}

\renewcommand{\.}{{}_{\!}} 

\renewcommand{\a}{\alpha}

\renewcommand{\c}{\gamma}

\newcommand{\e}{\varepsilon}
\newcommand{\f}{\varphi}

\renewcommand{\l}{\lambda}

\newcommand{\Om}{\Omega}

\newcommand{\s}{\sigma}

\newcommand{\cT}{\mathcal{T}}

\newcommand{\as}{\ \mbox{\raisebox{.085ex}{$:$}$\! =$} \ \;\!\.} 

\let\oldexists\exists
\renewcommand{\exists}{\oldexists \,} 
\let\oldforall\forall
\renewcommand{\forall}{\oldforall \,} 

\newcommand{\imp}{\ \Longrightarrow \ } 
\newcommand{\con}{\ \Longleftarrow \ } 

\newcommand{\st}{~|~} 

\newcommand{\ie}{\emph{i.\,e.\;}} 
\newcommand{\Oset}{\mbox{\large $\varnothing$}} 

\newcommand{\llist}[3]{#1_{#2} , \ldots , #1_{#3}} 
 
\newcommand{\inc}{\subseteq} 
\newcommand{\cni}{\supseteq} 
\newcommand{\setmin}{\,\! \raisebox{0.40ex}{$\smallsetminus$} \;\!\!} 

\newcommand{\cart}{\! \times \!} 

\newcommand{\1}[1]{\mathbf{1}_{\scriptscriptstyle \! #1 \!}} 

\newcommand{\rest}[2]{#1_{\mathbf{|} #2}} 
\newcommand{\inv}[1]{#1^{- \! 1}} 

\renewcommand{\to}{\longrightarrow} 
 
\newcommand{\seqN}[2]{\left( #1_{#2} \right)_{\! #2 \geq 0}} 
\newcommand{\seq}[3]{\left( #1 \right)_{\! #2 \geq #3}} 

\newcommand{\NN}{\mathbf{N}} 
\newcommand{\RR}{\mathbf{R}} 

\newcommand{\Rn}[1]{\mathbf{R}^{\! #1 \!}} 
 
\newcommand{\intr}[1]{\overset{\ _{\circ}}{#1}} 
\newcommand{\ri}[1]{\mathrm{ri} \! \left( #1 \right)} 
\newcommand{\clos}[1]{\, \overline{\! #1}} 
\newcommand{\rc}[1]{\mathrm{rc} \! \left( #1 \right)} 
\newcommand{\bd}[1]{\partial #1} 
\newcommand{\rb}[1]{\mathrm{rb} \! \left( #1 \right)} 

\renewcommand{\leq}{\leqslant} 
\renewcommand{\geq}{\geqslant} 

\newcommand{\Epi}[1]{\mathrm{Epi} \! \left( #1 \right)} 
\newcommand{\Epis}[1]{\mathrm{Epi_{s \,}} \!\! \left( #1 \right)} 

\let\oldint\int
\renewcommand{\int}[4]{\oldint_{\! #1}^{#2}{\!\!\!\! #3 \mathrm{d} #4}} 

\newcommand{\cLL}[1]{\mathcal{L}^{\mbox{\raisebox{.5ex}{$\scriptstyle #1$}} \!}} 
 
\newcommand{\goes}{\rightarrow} 

\newcommand{\lvec}[3]{\left( #1_{#2} , \ldots , #1_{#3} \right)} 
 
\renewcommand{\Im}[1]{\mathrm{Im} \! \left( #1 \right)} 

\renewcommand{\dim}[1]{\mathrm{dim} \! \left( #1 \right)} 
 
\newcommand{\Vect}[1]{\mathrm{Vect} \! \left( #1 \right)} 

\let\olddet\det
\renewcommand{\det}[1]{\olddet{\! \left( #1 \right)}} 
 
\newcommand{\Aff}[1]{\mathrm{Aff} \! \left( #1 \right)} 
\newcommand{\Conv}[1]{\mathrm{Conv} \! \left( #1 \right)} 

\newcommand{\twonorm}[1]{\left\| #1 \right\|_{_{\! 2}}} 

\newcommand{\scal}[2]{\left\langle #1 , #2 \right\rangle} 

\newcommand{\ppdereq}[2]{\frac{\partial^{2 \!} #1}{\partial #2^{2}}} 
 
\newcommand{\Cl}[1]{\mathrm{C}^{\mbox{\raisebox{.5ex}{$\scriptstyle #1$}} \!\!}} 

\begin{document}

\title[]{Functions with strictly convex epigraph}

\author{St\'{e}phane Simon}
\address{St\'{e}phane Simon, 
UMR 5127 du CNRS \& Universit\'{e} de Savoie, 
Laboratoire de math\'{e}matique, 
Campus scientifique, 
73376 Le Bourget-du-Lac Cedex, 
France}
\email{Stephane.Simon@univ-savoie.fr}

\author{Patrick Verovic}
\address{Patrick Verovic, 
UMR 5127 du CNRS \& Universit\'{e} de Savoie, 
Laboratoire de math\'{e}matique, 
Campus scientifique, 
73376 Le Bourget-du-Lac Cedex, 
France}
\email{Patrick.Verovic@univ-savoie.fr}

\date{\today}


\begin{abstract}
The aim of this paper is to provide a complete and simple characterization of functions with domain 
in a topological real vector space whose epigraph is strictly convex. 
\end{abstract}

\maketitle

\bigskip
\bigskip


\section*{Introduction} 

This paper is concerned with the relationship between strict convexity of functions defined 
over a domain in a topological real vector space and strict convexity of their epigraphs. 
A subset of a topological real vector space is said to be strictly convex if it is convex 
and if in addition there is no non-trivial segment in its boundary. 
This notion will be made more precise in Section~\ref{sec:preliminaries}. 

\bigskip

Even though strict convexity is less studied than convexity in the literature, there are nevertheless 
many fields where strict convexity of a subset of a topological real vector space is used. 
We give here three different examples which illustrate this geometric property. 

\bigskip

The first example, which is actually the starting point of the present work, 
concerns epigraphs of functions. 
It is a well-known result that a function has a convex epigraph if and only if it is convex. 
This is a bridge between geometric convexity and analytic convexity 
(see for example~\cite[Theorem~4.1, page~25]{Roc70}). 
Therefore, a natural question is to know whether the same equivalence holds 
when replacing ``convexity" by ``strict convexity". To the best of our knowledge, 
nothing has been studied about this issue in the litterature, 
even in the case when the domain of $\. f \.$ lies in $\Rn{n}$. 
This is why we propose to fill the gap in the present paper, 
not only in $\Rn{n}$ but in the general framework of topological real vector spaces. 
This is done in the Main~Theorem that we state in Section~\ref{sec:main-thm}. 

\bigskip

The second example deals with strict convexity of the unit ball in a normed real vector space 
(in that case, the norm itself is sometimes called strictly convex, which is unfortunate). 
This property is equivalent to saying that there exists a real number ${p > \;\!\! 1 \.}$ 
such that the $p$-th power of the norm is a strictly convex function 
(we may see Theorem~11.1 in~\cite[page~110]{Car04}), and this is actually equivalent 
to the strict convexity of the epigraph of this function as the Main~Theorem will show. 

\smallskip

It is important to work with such norms since they yield interesting properties in functional analysis. 
For instance, given a Banach space $E$ with strictly convex unit ball, 
any non-empty family of commutative non-expansive mappings from a non-empty closed convex 
and weakly compact subset of $E$ into itself has a common fixed point (see for example~\cite{BKS68}). 

\smallskip

Nevertheless, if the unit ball of a normed vector space is not strictly convex, this may be offset 
in at least two different ways. Indeed, any reflexive Banach space can be endowed with an equivalent norm whose unit ball is strictly convex (see for example~\cite{Lin66}). 
On the other hand, any separable Banach space can be endowed with an equivalent norm which is smooth and whose unit ball is strictly convex (see for example~\cite[page~33]{JohLin01}). 

\bigskip

The third example concerns optimization problems---more precisely, the relationship between 
strict convexity and uniqueness of minimizers. 
When dealing with an optimization problem on a topological real vector space, the search for a value 
of the variable where the cost function achieves a minimum is much more easier in case this function 
and the constraint set are both convex (see for example~\cite{EkeTem99}). 
Moreover, if the cost function is strictly convex, such a minimizer is then unique. 
On the other hand, if the constraint set is both strictly convex and given by the epigraph 
of a function (we shall see in which case this is possible owing to the Main~Theorem), 
and if the cost function has no minimum over the whole space, then such a minimizer is unique too. 

\bigskip

As we may notice throughout these three examples, it is of great importance to know 
whether the epigraph of a function is strictly convex or not. 

\bigskip

Of course, for a function defined over $\Rn{n}$, the strict convexity of its epigraph is merely 
equivalent to being strictly convex. But what happens for a function with an \emph{arbitrary} domain 
which lies in an \emph{arbitrary} topological real vector space? 

\bigskip

In order to give a complete answer to this question and deduce some of its consequences 
in Section~\ref{sec:main-thm}, we shall examine two issues in Section~\ref{sec:preliminaries}: 
topological aspects of epigraphs of functions defined on any topological space on the one hand, 
and the notion of strict convexity for sets in general topological real vector spaces 
(no matter which dimension they have or whether they are Hausdorff) on the other hand. 

\bigskip

Finally, Section~\ref{sec:proofs} is devoted to the proofs of all the results 
we mention in the previous sections. 

\bigskip
\bigskip


\section{Motivations, Main Theorem and consequences} \label{sec:main-thm} 

The relationship between convexity of sets and convexity of functions 
is given by the following well-known result that is quite easy to prove. 

\bigskip

\begin{proposition} \label{prop:convex-epi} 
   Let $C \!$ be a subset of a real vector space $V \!\!$ and $\. f \. : C \to \RR$ a function. 
   Then we have the equivalence 
   $$
   C \! \ \mbox{and} \ \. f \! \ \mbox{are both convex} 
   \qquad \iff \qquad 
   {\Epi{f}} \;\!\! \ \mbox{is convex}~.
   $$ 
\end{proposition}

\medskip

This is a geometric way of characterizing the convexity of a function by looking at its epigraph. 

\bigskip

Such a property naturally raises the issue of studying what happens when convexity 
is replaced by strict convexity whose meaning will be given in Definition~\ref{def:strict-cvx} 
(real vector spaces being of course replaced by 
\emph{arbitrary}---possibly non-Hausdorff---topological real vector spaces). 

\bigskip

At first sight, we may believe that for a function defined over a general topological 
real vector space, strict convexity of its epigraph is merely equivalent for the function 
to be strictly convex in the usual sense. 
But this is \emph{false} as we can observe in the following example. 

\bigskip

\begin{example} \label{ex:epi-not-strictly-cvx} 
Consider the real vector space 
$V \!\. \as \Cl{0 \,}(\RR , \RR) \cap \cLL{2}(\RR , \RR) \inc \Rn{\RR}$ 
endowed with the topology $\cT \.$ of pointwise convergence 
(this is nothning else than the product topology, which is therefore Hausdorff), 
and let $\. f \. : C \. \as V \! \to \RR$ be the function defined by 
$\. f \. (u) \. \as \!\. \twonorm{u}^{\. 2} \!$. 

\smallskip

On the one hand, $\. f \.$ is strictly convex since for any $u \in V \!$ its Hessian at $u$ 
with respect to the norm ${\twonorm{\cdot}} \!$ on $V \!$ 
is equal to ${2 \. {\scal{\cdot \,}{\cdot}} \.}$, and hence positive definite. 

\smallskip

On the other hand, whereas $\;\!\! (0 , 0) \;\!\!$ and $\;\!\! (0 , 2) \;\!\!$ 
are in the epigraph of $\. f \!$, their midpoint $\;\!\! (0 , \;\!\! 1 \:\!\. ) \;\!\!$ 
does not belong to $\ri{{\Epi{f}} \.} = \intr{\wideparen{\Epi{f}}} \.$, 
that is, ${{\Epi{f}} \;\!\!}$ is not a neighborhood of $\;\!\! (0 , \;\!\! 1 \:\!\. ) \;\!\!$ 
for the product topology on $V \. \cart \RR$ 
as we can check with the sequence ${\;\!\! {\seq{u_{n}}{n}{\. 1 \.}}}$ of $V \!$ defined by 
$$
u_{n \.}(x) \. \as \!\. \left\{ 
\begin{array}{cl} 
   2 \sqrt{n x - n^{2} \. + \;\!\! 1 \. /  \. n} & 
   \quad \mbox{for} \quad x \in [n - \;\!\! 1 \. /  \. n^{2} \! \, , \, n] \.~, \\ 
   2 \. / \! \sqrt{n} & \quad \mbox{for} \quad x \in [n \, , \, 2 n] \.~, \\ 
   2 \sqrt{2 n^{2} \. - n x + \;\!\! 1 \. /  \. n} & 
   \quad \mbox{for} \quad x \in [2 n \, , \, 2 n + \;\!\! 1 \. /  \. n^{2}] \.~, \quad \mbox{and} \\ 
   0 & \quad \mbox{for} \quad x \leq n - \;\!\! 1 \. /  \. n^{2} \! 
   \quad \mbox{or} \quad x \geq 2 n + \;\!\! 1 \. /  \. n^{2} \,\!\!\!~, 
\end{array} 
\right.
$$ 
which converges to zero with respect to $\cT$ but satisfies 
${\;\!\! (u_{n} , \;\!\! 1 \:\!\. ) \not \in {\Epi{f}} \;\!\!}$ 
for any $n \geq \;\!\! 1 \.$ since one has 
$\. f \. (u_{n}) \geq \:\! \:\!\! f \. (u_{n} \! \times \! \1{[n \, , \, 2 n] \,}) 
= \:\! 4 > \;\!\! 1$. 
This proves that $\Epi{f} \inc V \. \cart \RR$ is not strictly convex. 
\end{example}

\bigskip

Even in the Hausdorff finite-dimensional case, things are not as simple as they seem. 
Indeed, if we consider the open disc 
$C \. \as \:\!\! \{ (x , y) \in \Rn{2} \st x^{2} \! + y^{2} \. < \;\!\! 1 \:\!\. \} \.$, 
the function $\. f \. : C \to \RR$ defined by 
$\. f \. (x , y) \. \as x^{2} \! + y^{2} \!$ is strictly convex but its epigraph is not. 

\bigskip

Moreover, either in the non-Hausdorff finite-dimensional case 
or in the infinite-dimensional case, convexity does not always implies continuity. 
In contrast and among other things, we shall see that \emph{strict convexity} of the epigraph 
\emph{does always} insure continuity of the function. 

\bigskip

So, the question of finding a complete characterization of functions with domain 
in a topological real vector space whose epigraph is strictly convex does deserve our attention. 

\bigskip

This characterization is described as follows, where $\mathrm{rb}$ stands for the relative boundary 
(see Definition \ref{def:ri+rc+rb}). 

\bigskip

\begin{maintheorem} 
   Let $C \!$ be a subset of a topological real vector space $V \!\!$ and 
   ${\. f \. : C \to \RR}$ a function. 
   Then we have the following equivalence: 
   \[
   \left( \!\!\! 
   \begin{array}{c} 
      \vspace{5pt} 
      C \! \ \mbox{is convex and open in} \ {\Aff{C}} \.~, \\ 
      \vspace{5pt} 
      f \! \ \mbox{is strictly convex and continuous}~, \\ 
      \forall x_{0} \in \rb{C} \inc \:\! \clos{C} \. , \ \ \. 
      f \. (x) \ \to \ \:\! {+\infty} \quad \mbox{as} \quad x \ \to \ x_{0} 
   \end{array} \!\! 
   \right) \!\!\: 
   \ \iff \ 
   {\Epi{f}} \;\!\! \ \mbox{is strictly convex}~.
   \] 
\end{maintheorem}

\medskip

\begin{remark*} 
It is to be noticed that this equivalence is still true if ``continuous'' 
is changed into ``locally bounded from above''. 
\end{remark*}

\bigskip

The proof, that we postpone until Section~\ref{sec:proofs}, 
splits into the direct implication and its converse.

\smallskip

The direct implication is the consequence of three main facts. 
The first one is the convexity of ${{\Epi{f}} \;\!\!}$ given by the convexity of $\. f \!$. 
The second one is the closeness of the epigraph of $\. f \.$ in $\Aff{C} \cart \RR$ 
due to both the continuity of the $\. f \.$ and its behavior near the boundary of $C \.$, 
which insures that any segment whose end points are in the boundary of ${{\Epi{f}} \;\!\!}$ 
is contained in ${{\Epi{f}} \.}$. 
The third one is the property that any open segment whose endpoints 
are in the boundary of ${{\Epi{f}} \;\!\!}$ actually lies inside the interior of ${{\Epi{f}} \;\!\!}$ 
in $\Aff{C} \cart \RR$ as a result of the strict convexity of $\. f \.$ and the two previous facts. 

\smallskip

As for the converse implication, there are four main things to be used. 
The first one is the convexity of both $C \.$ and $\. f \.$ given by the convexity of ${{\Epi{f}} \.}$. 
The second one is the openness of $C \.$ in ${{\Aff{C}} \;\!\!}$ as a consequence for the epigraph 
not to contain vertical segments in its boundary. 
The third one is the fact that the interior of ${{\Epi{f}} \;\!\!}$ in $\Aff{C} \cart \RR$ lies 
inside the strict epigraph of $\. f \!$. 
The fourth one is the property for $\. f \.$ to be locally bounded on some non-empty open set in~$C \.$ 
as a result of the non-emptyness of the interior of ${{\Epi{f}} \;\!\!}$ in $\Aff{C} \cart \RR$. 
All these properties yield the continuity of $\. f \.$ and give the behavior of $\. f \.$ 
near the boundary of $C \.$. 

\bigskip

Before giving some consequences of this result, 
all of whose will also be proved in Section~\ref{sec:proofs}, let us just show on a simple example 
how it may be usefull for checking strict convexity of the epigraph of a function. 

\bigskip

\begin{example} \label{ex:epi-strictly-cvx} 
Consider the open convex subset 
${C \. \as \;\!\! {]{- \. 1} \:\!\. , \;\!\! 1 \:\!\. [} 
\cart {]{- \. 1} \:\!\. , \;\!\! 1 \:\!\. [} \;\!\!}$ 
of the topological real vector space ${V \! \as \Rn{2}}$ (with its usual topology), 
and let ${\. f \. : C \to \RR}$ be the smooth function defined by 
${\. f \. (x , y) \. 
\as \;\!\! 1 \. / \. [ \. ( \:\!\. 1 \;\!\! - x^{2}) \. ( \:\!\. 1 \;\!\! - y^{2}) \. ] \.}$. 

\smallskip

For any point ${\;\!\! (x , y) \in \:\! C \.}$, we then compute 
${\disp \ppdereq{f}{x}(x , y) 
= \:\! 2 \! \times \! \frac{1 \;\!\! + 3 x^{2}}{(1 \;\!\! - x^{2}) \. (1 \;\!\! - y^{2})} > 0}$, 
and the Hessian matrix of $\. f \.$ at $(x , y)$ has a determinant which is equal to 
$$
4 (5 x^{2} y^{2} \! + 3 y^{2} \! + 3 x^{2} \! + \;\!\! 1 \:\!\. ) \! 
/ \. [ \. (x - \;\!\! 1 \:\!\. )^{\. 4} (x + \;\!\! 1 \:\!\. )^{\. 4} 
(y - \;\!\! 1 \:\!\. )^{\. 4} (y + \;\!\! 1 \:\!\. )^{\. 4}] \ > \ 0~.
$$ 


The function $\. f \.$ is therefore strictly convex and hence the Main~Theorem 
insures that its epigraph is strictly convex since we have $\Aff{C} = \Rn{2}$ 
and $\. f \. (x , y) \to \:\! {+\infty}$ 
as $\;\!\! (x , y) \;\!\!$ converges to any point $\;\!\! (x_{0} , y_{0}) \in \bd{C} \.$. 
\end{example}

\bigskip

The first consequence of the Main~Theorem is obtained by taking $C \as V \!\.$. 

\bigskip

\begin{proposition} \label{prop:sc-intr-epi} 
   Given a strictly convex function $\. f \. : V \! \to \RR$ 
   defined on a topological real vector space $V \!\!$, we have the equivalence 
   
   \smallskip
   
   \begin{center} 
      ${{\Epi{f}} \;\!\!}$ is strictly convex 
      \qquad $\iff$ \qquad 
      ${{\Epi{f}} \;\!\!}$ has a non-empty interior in $V \! \cart \RR$~. 
   \end{center} 
\end{proposition}

\bigskip

Another use of the Main~Theorem is related to the property for a subset $C \.$ 
of a real vector space $V \!$ to be convex if and only if all its intersections 
with the straight lines of $V \!$ are convex. 

\smallskip

Indeed, let us recall the following easy-to-prove result about convex functions. 

\bigskip

\begin{proposition} 
   For any subset $C \!$ of a real vector space $V \!\!$ 
   and any function ${\. f \. : C \to \RR}$, we have (1)$\iff$(2)$\iff$(3) with 
   
   \smallskip
   
   \begin{enumerate}[(1)]
      \item ${{\Epi{f}} \;\!\!}$ is convex, 
      
      \medskip
      
      \item ${{\Epi{\rest{f}{C \cap G \.}}} \;\!\!}$ is convex 
      for any affine subspace $G \!$ of $V \!\!$, and 
      
      \medskip
      
      \item ${{\Epi{\rest{f}{C \cap L \.}}} \;\!\!}$ is convex 
      for any straight line $L \.$ of $V \!\!$. 
   \end{enumerate} 
\end{proposition}

\bigskip

Then, a natural question is to know whether these equivalences are still true 
when replacing convexity by strict convexity. 

\smallskip

Here is the answer. 

\bigskip

\begin{proposition} \label{prop:epi-rest} 
   For any subset $C \!$ of a topological real vector space $V \!\!$ 
   and any function ${\. f \. : C \to \RR}$, we have (1)$\iff$(2)$\imp$(3) with 
   
   \smallskip
   
   \begin{enumerate}[(1)]
      \item ${{\Epi{f}} \;\!\!}$ is strictly convex, 
      
      \medskip
      
      \item ${{\Epi{\rest{f}{C \cap G \.}}} \;\!\!}$ is strictly convex 
      for any affine subspace $G \!$ of $V \!\!$, and 
      
      \medskip
      
      \item ${{\Epi{\rest{f}{C \cap L \.}}} \;\!\!}$ is strictly convex 
      for any straight line $L \.$ in $V \!\!$. 
   \end{enumerate} 
\end{proposition}

\bigskip

It is to be noticed that the implication~(3)$\imp$(1) 
in Proposition~\ref{prop:epi-rest} is \emph{not} true.
Indeed, the function $\. f \. : V \! \to \RR$ that we considered 
in Example~\ref{ex:epi-not-strictly-cvx} has an epigraph which is not strictly convex 
whereas for any vectors $u_{0} , w \in V \!$ 
with $w \neq 0$ the function $\f : \RR \to \RR$ defined by 
${\f(t) \. \as \:\!\! f \. (u_{0} + t w) = 
\twonorm{w}^{\. 2} \! t^{2} \! + \;\!\! 2 \! \scal{u_{0}}{w} \! t + \twonorm{u_{0}}^{\. 2} \!}$ 
is obviously strictly convex and continuous. Therefore, ${{\Epi{\rest{f}{u_{0} + \RR w \.}}} \;\!\!}$ 
is strictly convex according to the Main~Theorem 
and since the map ${\c : \RR \to u_{0} + \RR w}$ defined by ${\c(t) \. \as u_{0} + t w}$ 
is a homeomorphism. 
This last point is a consequence of Theorem~2 in~\cite[Chapitre~I, page~14]{BouEVT81} 
since $u_{0} + \RR w$ is a finite-dimensional affine space whose subspace topology is Hausdorff 
(as is the topology on $V$). 

\bigskip

Nevertheless, in case $V \!$ is equal to the canonical topological real vector space $\Rn{n}$, 
this implication is true as a consequence of the Main~Theorem. 

\bigskip

\begin{proposition} \label{prop:epi-finite-dim} 
   Given a subset $C \!$ of $\Rn{n} \.$ and a function ${\. f \. : C \to \RR}$, 
   the following properties are equivalent: 
   
   \smallskip
   
   \begin{enumerate}[(1)]
      \item ${{\Epi{f}} \;\!\!}$ is strictly convex. 
      
      \medskip
      
      \item ${{\Epi{\rest{f}{C \cap L \.}}} \;\!\!}$ is strictly convex 
      for any straight line $L \.$ in $\Rn{n} \.$. 
   \end{enumerate} 
\end{proposition}

\bigskip

As a straightforward consequence of Proposition~\ref{prop:epi-finite-dim}, 
we obtain in particular the following classical result. 

\bigskip

\begin{consequence*} 
   Any function $\. f \. : \Rn{n} \to \RR$ satisfies the equivalence 
   
   \smallskip
   
   \begin{center} 
      ${{\Epi{f}} \;\!\!}$ is strictly convex \qquad $\iff$ \qquad $\. f \!$ is strictly convex~. 
   \end{center} 
\end{consequence*}

\bigskip
\bigskip


\section{Preliminaries} \label{sec:preliminaries} 

In order to make precise the terms used in the previous section, 
and before proving in Section~\ref{sec:proofs} the Main~Theorem and its consequences, 
we have to give here some definitions and properties about the epigraph of a function 
and the notion of strict convexity. 

\bigskip

\subsection{Epigraphs} \label{sec:epigraphs} 

We begin by recalling the definitions of the epigraph and the strict epigraph for a general function. 

\bigskip

\begin{definition} 
   Given a set $X \.$ and a function $\. f \. : X \! \to \RR$, 
   the \emph{epigraph} of $\. f \.$ is defined by 
   $$
   \Epi{f} \!\. \as \:\!\! \{ (x , r) \in X \. \cart \RR \st \. f \. (x) \leq \:\! r \} \.~,
   $$ 
   whereas its \emph{strict epigraph} is defined by 
   $$
   \Epis{f} \!\. \as \:\!\! \{ (x , r) \in X \. \cart \RR \st \. f \. (x) < \:\! r \} \.~.
   $$ 
\end{definition}

\medskip

\begin{remark} \label{rem:epi-transfo} 
It is straightforward that these two sets satisfy the relation 
$$
\Epis{f} \ = \;\!\. \ (X \. \cart \RR) \! \setmin \s(\Epi{- \! f}) \.~,
$$ 
where $\s$ denotes the involution of $X \. \cart \RR$ 
defined by $\s(x , r) \. \as \;\!\! (x , {-r}) \.$. 
\end{remark}

\bigskip

From now on, $X \.$ will denote a \emph{topological} space 
and we shall give a list of useful properties of the epigraph of a function 
with domain in $X \.$ (Proposition~\ref{prop:loc-bounded} to Proposition~\ref{prop:epi-infinity}) 
that we will need in the sequel (we may refer to~\cite[pages~34 and~123]{HKS05}). 

\bigskip

\begin{proposition} \label{prop:loc-bounded} 
   Given a subset $S \!$ of a topological space $X \!\.$ and a function $\. f \. : S \to \RR$ 
   such that ${{\Epi{f}} \;\!\!}$ has a non-empty interior in $X \. \cart \RR$, 
   there exists a non-empty open set $U \!\.$ in $X \!\.$ with $U \. \inc S \!$ 
   on which $\. f \!$ is bounded from above. 
\end{proposition}

\medskip

\begin{proof}~\\ 

\vspace{-17pt}

Given ${\;\!\! (x_{0} , t_{0}) \in \intr{\wideparen{\Epi{f}}} \.}$, 
there exists an open set $U \!$ in $X \.$ such that 
one has $x_{0} \in U \!$ and ${U \. \cart \. \{ t_{0} \} \. \inc {\Epi{f}} \;\!\!}$ 
since ${{\Epi{f}} \;\!\!}$ is a neighborhood of ${\;\!\! (x_{0} , t_{0}) \;\!\!}$ in $X \. \cart \RR$. 
Therefore, for any $x \in U \!$, we get $\. f \. (x) \leq \:\! t_{0}$. 
\end{proof}

\bigskip

\begin{proposition} \label{prop:epi-interior} 
   Given a function $\. f \. : X \! \to \RR$ defined on a topological space $X \!$, 
   a point $x_{0} \in \. X \!\.$ and a number $r_{0} \in \RR$, 
   the following properties are equivalent: 
   
   \vspace{-4pt}
   
   \begin{enumerate}[(1)]
      \item $(x_{0} , r_{0}) \in \intr{\wideparen{\Epi{f}}} \.$. 
      
      \vspace{-3pt}
      
      \item $\exists r < r_{0}, \ x_{0} \in \intr{\wideparen{\inv{f}( \. {-\infty} , r)}} \.$. 
      
      \vspace{-3pt}
      
      \item $\exists s < r_{0}, \ \. \{ x_{0} \} \. \cart [s , {+\infty}) 
      \inc \intr{\wideparen{\Epi{f}}} \.$. 
   \end{enumerate} 
\end{proposition}

\medskip

This is a characterization of the interior of the epigraph of a function. 

\medskip

\begin{proof}~\\ 

\vspace{-17pt}

\textsf{Point~1 $\imp$ Point~2.} 
Assume we have ${\;\!\! (x_{0} , r_{0}) \in \intr{\wideparen{\Epi{f}}} \.}$. 

\smallskip

Then there exist a neighborhood $V \!$ of $x_{0}$ in $X \.$ and a number $\e > 0$ 
that satisfy the inclusion 
${V \. \cart [r_{0} - \;\!\! 2 \e \, , \, r_{0} + \;\!\! 2 \e] \inc {\Epi{f}} \.}$, 
from which we get $\. f \. (x) \leq \:\! r_{0} - \;\!\! 2 \e$ for any $x \in V \!\!$, 
or equivalently ${V \. \inc \inv{f}( \. {-\infty} , r) \;\!\!}$ 
with $r \as r_{0} - \e < r_{0}$. 

\smallskip

This proves that $x_{0}$ belongs to the interior 
of ${\inv{f}( \. {-\infty} , r) \;\!\!}$ in $\RR$. 


\textsf{Point~2 $\imp$ Point~3.} 
Assume that we have ${x_{0} \in \intr{\wideparen{\inv{f}( \. {-\infty} , r)}} \;\!\!}$ 
for some $r < r_{0}$. 

\smallskip

Therefore, there is a neighborhood $V \!$ of $x_{0}$ in $X \.$ 
that satisfies ${V \. \inc \inv{f}( \. {-\infty} , r) \inc \inv{f}( \. {-\infty} , r] \.}$, 
which yields ${V \. \cart [r , {+\infty}) \inc {\Epi{f}} \.}$, and hence the inclusion 
${\. \{ x_{0} \} \. \cart [s , {+\infty}) \inc \intr{\wideparen{\Epi{f}}} \;\!\!}$ holds 
with ${s \:\!\. \as \;\!\! (r + r_{0}) \! / \. 2}$ 
since we have ${s \in (r , r_{0}) \;\!\!}$ and since the interval ${\;\!\! [r , {+\infty}) \;\!\!}$ 
is a neighborhood in $\RR$ of any number $\tau \in [s , {+\infty}) \.$. 

\medskip

\textsf{Point~3 $\imp$ Point~1.} 
This is clear. 
\end{proof}

\bigskip

\begin{proposition} \label{prop:epi-intr+clos} 
   Any function $\. f \. : X \! \to \RR$ defined on a topological space $X \!\.$ 
   satisfies the following properties: 
   
   \vspace{-2pt}
   
   \begin{enumerate}[(1)]
      \item $\intr{\wideparen{\Epi{f}}} = \ \!\! \intr{\wideparen{\Epis{f}}} \.$. 
      
      
      \item ${\intr{\wideparen{\Epi{f}}} \cap ( \. \{ x \} \. \cart \RR) 
      \inc \. \{ x \} \. \cart (f \. (x) , {+\infty}) \;\!\!}$ for any $x \in \. X \!$. 
   \end{enumerate} 
\end{proposition}

\medskip

This property gives a topological relationship between 
the epigraph and the strict epigraph of a function. 

\medskip

\begin{proof}~\\ 

\vspace{-17pt}

\textsf{Point~1.} 
Given a point ${\;\!\! (x , r) \in \intr{\wideparen{\Epi{f}}} \.}$, there exists $r_{0} < r$ 
that satisfies ${\. \{ x \} \. \cart [r_{0} , {+\infty}) \inc {\Epi{f}} \.}$, 
which yields in particular ${\. f \. (x) \leq \:\! r_{0}}$, and hence ${\. f \. (x) < \:\! r}$. 
This proves ${\intr{\wideparen{\Epi{f}}} \inc {\Epis{f}} \.}$, which implies 
${\intr{\wideparen{\Epi{f}}} \inc \intr{\wideparen{\Epis{f}}} \;\!\!}$ 
by taking the interiors in the product space $X \. \cart \RR$. 


Conversely, the obvious inclusion ${\Epis{f} \inc {\Epi{f}} \;\!\!}$ yields 
$\intr{\wideparen{\Epis{f}}} \inc \intr{\wideparen{\Epi{f}}} \.$. 

\smallskip

\textsf{Point~2.} 
By Point~1 above, we have ${\intr{\wideparen{\Epi{f}}} \inc {\Epis{f}} \.}$, and hence 
$\intr{\wideparen{\Epi{f}}} \cap ( \. \{ x \} \. \cart \RR) 
\inc \Epis{f} \cap ( \. \{ x \} \. \cart \RR) 
= \;\!\! \{ x \} \. \cart (f \. (x) , {+\infty}) \;\!\!$ for any $x \in \. X \.$. 
\end{proof}

\bigskip

\begin{proposition} \label{prop:epi-usc} 
   Given a function $\. f \. : X \! \to \RR$ defined on a topological space $X \!\.$ 
   and a point $x_{0} \in \. X \!$, the following properties are equivalent: 
   
   \smallskip
   
   \begin{enumerate}[(1)]
      \item $\. f \!$ is upper semi-continuous at $x_{0}$. 
      
      \vspace{-2pt}
      
      \item $\{ x_{0} \} \. \cart (f \. (x_{0}) , {+\infty}) 
      = \ \!\! \intr{\wideparen{\Epi{f}}} \cap ( \. \{ x_{0} \} \. \cart \RR) \.$. 
   \end{enumerate} 
\end{proposition}

\medskip

This is a geometric characterization of the upper semi-continuity 
of a function in terms of its epigraph. 

\medskip

\begin{proof}~\\ 
\textsf{Point~1 $\imp$ Point~2.} 
Given an arbitrary $\e > 0$, there exists a neighborhood $V \!$ of $x_{0}$ in $X \.$ that satisfies 
$\. f \. (x) \leq \:\! \:\!\! f \. (x_{0}) \;\!\! + \e \. / \. 2$ for any $x \in V \!\.$. 

\vspace{-5pt}

Thus, we have 
${V \. \cart [f \. (x_{0}) \;\!\! + \e \. / \. 2 \, , {+\infty}) \inc {\Epi{f}} \.}$, 
and hence ${\;\!\! (x_{0} , \. f \. (x_{0}) \;\!\! + \e) \in \intr{\wideparen{\Epi{f}}} \.}$. 

\smallskip

So we proved the direct inclusion $\inc$. 

\smallskip

The reverse inclusion $\cni$ is straightforward 
by Point~2 in Proposition~\ref{prop:epi-intr+clos}. 

\bigskip

\textsf{Point~2 $\imp$ Point~1.} 
Conversely, given an arbitrary number $\e > 0$, we have 
$\;\!\! (x_{0} , \. f \. (x_{0}) \;\!\! + \e) 
\in \. \{ x_{0} \} \. \cart (f \. (x_{0}) , {+\infty}) \inc \intr{\wideparen{\Epi{f}}} \.$. 

\medskip

Thus, there exists a neighborhood $V \!$ of $x_{0}$ in $X \.$ satisfying 
${V \. \cart \{ f \. (x_{0}) \;\!\! + \e \} \. \inc {\Epi{f}} \.}$, 
which yields $\. f \. (x) \leq \:\! \:\!\! f \. (x_{0}) \;\!\! + \e$ for any $x \in V \!\.$. 
\end{proof}

\bigskip

\begin{proposition} \label{prop:epi-infinity} 
   Given a subset $A \.$ of a topological space $X \!$, a point $x_{0} \in \clos{A} \.$ 
   and a function ${\. f \. : A \to \RR}$, we have the equivalence 
   $$
   \big( f \. (x) \ \to \ \:\! {+\infty} \quad \mbox{as} \quad x \ \to \ x_{0} \big) 
   \qquad \iff \qquad 
   \big( \. (x_{0} , r) \not \in \! \clos{\, \Epi{f}} \;\!\! 
   \quad \mbox{for any} \quad r \in \RR \big) \.~.
   $$ 
\end{proposition}

\smallskip

This is a characterization of the closure of the epigraph of a function. 

\medskip

\begin{proof}~\\ 
\textasteriskcentered \ ($\imp$) 
Given $r \in \RR$, there exists a neighborhood $V \!$ of $x_{0}$ in $X \.$ 
such that we have in particular the inclusion $\. f \. (V) \inc [r + \;\!\! 2 \, , \, {+\infty}) \.$. 

\smallskip

Thus, we get 
${\;\!\! (V \. \cart ( \. {-\infty} \, , \, r + \;\!\! 1 \:\!\. ] \. ) \cap \Epi{f} = \Oset \.}$, 
which shows that $\;\!\! (x_{0} , r) \;\!\!$ does not belong to ${\! \clos{\, \Epi{f}} \;\!\!}$ 
since ${V \. \cart ( \. {-\infty} \, , \, r + \;\!\! 1 \:\!\. ] \;\!\!}$ 
is a neighborhood of $\;\!\! (x_{0} , r) \;\!\!$ in $X \. \cart \RR$. 

\medskip

\textasteriskcentered \ ($\con$) 
Given $r \in \RR$, there exist a neighborhood $V \!$ of $x_{0}$ in $X \.$ and a real number $\e > 0$ 
such that ${V \. \cart [r - \e \, , \, r + \e] \;\!\!}$ does not meet ${{\Epi{f}} \.}$. 

\smallskip

Since we have $r \in [r - \e \, , \, r + \e] \.$, 
this yields ${\. f \. (V) \cap ( \. {-\infty} , r) = \Oset \.}$, 
which is equivalent to the inclusion $\. f \. (V) \inc (r , {+\infty}) \.$. 

\smallskip

So, we have proved $\. f \. (x) \to \:\! {+\infty}$ as $x \to x_{0}$. 
\end{proof}

\bigskip

\subsection{Strict convexity} \label{sec:strict-convexity} 

In this subsection, we first recall the definitions of the relative interior and the relative closure 
of a set in a topological real vector space since they underly strict convexity, 
and then we establish a couple of useful properties needed in Section~\ref{sec:proofs}. 

\bigskip

We begin with two basic notions in affine geometry: the affine hull and convex sets 
(see for example~\cite{Roc70} and~\cite{Web94}). 

\bigskip

\begin{definition} \label{def:aff-hull} 
   The \emph{affine hull} ${{\Aff{S}} \;\!\!}$ of a subset $S$ of a real vector space $V \!$ 
   is the smallest affine subspace of $V \!$ which contains $S \.$. 
\end{definition}

\bigskip

So, for any subsets $A$ and $B$ of $V \!$ satisfying $A \inc B$, 
we obviously have ${\Aff{A} \inc {\Aff{B}} \.}$. 

\bigbreak

\begin{proposition} \label{prop:aff-cart} 
   Given real vector spaces $V \!\!$ and $W \!\!$, 
   any subsets $A \inc V \!\!$ and $B \inc W \!\!$ satisfy 
   \[
   \Aff{A \cart B} \ = \ \Aff{A} \cart {\Aff{B}} \.~.
   \] 
\end{proposition}

\smallskip

\begin{proof}~\\ 
\textasteriskcentered \ 
For any ${x , y \in \. B}$ with $x \neq y$, 
the straight line $L$ passing through $x$ and $y$ lies in ${{\Aff{B}} \.}$, 
and hence for any $a \in \. A$ the straight line $\. \{ a \} \cart L$ of $V \. \cart W \!$ 
lies in ${{\Aff{A \cart B}} \;\!\!}$ 
since it contains the points $\;\!\! (a , x) , (a , y) \in A \cart B$. 
Therefore, we get ${A \cart \Aff{B} \inc {\Aff{A \cart B}} \.}$. 

\smallskip

On the other hand, we also have ${\Aff{A} \cart B \inc {\Aff{A \cart B}} \;\!\!}$ 
by the same reasoning. 

\smallskip

These two inclusions yield 
\[
\Aff{A} \cart \Aff{B} \ \inc \ \Aff{\Aff{A} \cart B} \ \inc \ \Aff{{\Aff{A \cart B}} \.} 
\ = \ {\Aff{A \cart B}} \.~.
\] 

\smallskip

\textasteriskcentered \ 
Conversely, since ${\Aff{A} \cart {\Aff{B}} \;\!\!}$ 
is an affine subspace of ${V \. \cart W \!}$ 
which contains ${A \cart B}$, we obviously have ${\Aff{A \cart B} \inc \Aff{A} \cart {\Aff{B}} \.}$. 
\end{proof}

\bigskip

\begin{definition} 
   Given points $x$ and $y$ in a real vector space $V \!\!$, the set 
   \[
   [x , y] \. \as \:\!\! \{ \. ( \:\!\. 1 \;\!\! - t) x + t y \st t \in [0 , \;\!\! 1 \:\!\. ] \. \}
   \] 
   is called the \emph{(closed) line segment} between $x$ and $y$, 
   whereas the set 
   \[
   {]x , y[} \. \as \;\!\! [x , y] \. \setmin \{ x , y \}
   \] 
   is called the \emph{open line segment} 
   between $x$ and $y$ (the latter set is therefore empty in case one has $x = y$). 
   
   \smallskip
   
   A subset $C \.$ of $V \!$ is said to be \emph{convex} 
   if we have ${\;\!\! [x , y] \inc \:\! C \.}$ for all $x , y \in C \.$. 
\end{definition}

\bigskip

In other words, $C \.$ is convex if and only if 
its intersection with any straight line $L$ in $V \!$ is an ``interval'' of $L$. 

\bigskip

From now on and throughout the section, $V \!$ will denote a \emph{topological} real vector space. 

\bigskip

Before we go on, let us just point out some facts. 

\bigskip

\begin{remark} \label{rem:0-tvs} 
~
\begin{enumerate}[1)]
\item Given a neighborhood $U \!$ of the origin in $V \!\!$, the following properties hold: 

\smallskip

\begin{enumerate}[(a)]
   \item For any vector $x \in V \!\!$, there exists $\e > 0$ 
   such that we have ${\;\!\! [{-\e} , \e \;\!\. ] x \inc U \!}$. 
   
   \smallskip
   
   \item For any vector ${x \in V \!\!}$, there exists ${\l > 0}$ such that we have ${x \in \l U \!}$ 
   (the set $U \!$ is then said to be \emph{absorbing}). 
   
   \smallskip
   
   \item We have $\Vect{U} = V \!\.$. 
\end{enumerate}

\smallskip

Indeed, Point~a is easy to prove 
and the implications \, a$\imp$b$\imp$c \, are straightforward. 

\medskip

\item Given a finite-dimensional real vector space $W \!\!$, 
there exists a \emph{unique} topological real vector space structure on $W \!$ which is Hausdorff. 
Endowed with this structure, $W \!$ is then isomorphic 
to the canonical topological real vector space $\Rn{n}$, where $n$ denotes the dimension of $W \!$ 
(see Theorem~2 in~\cite[Chapitre~I, page~14]{BouEVT81}). 
\end{enumerate} 
\end{remark}

\bigbreak

\begin{definition} \label{def:ri+rc+rb} 
   Let $S$ be a subset of a topological real vector space $V \!\.$. 
   
   \smallskip
   
   \begin{enumerate}[(1)]
      \item The \emph{relative interior} ${{\ri{S}} \;\!\!}$ of $S$ is the interior of $S$ 
      with respect to the relative topology of ${{\Aff{S}} \.}$. 
      
      \smallskip
      
      \item The \emph{relative closure} ${{\rc{S}} \;\!\!}$ of $S$ is the closure of $S$ 
      with respect to the relative topology of ${{\Aff{S}} \.}$. 
      
      \smallskip
      
      \item The \emph{relative boundary} ${{\rb{S}} \;\!\!}$ of $S$ is the boundary of $S$ 
      with respect to the relative topology of ${{\Aff{S}} \;\!\!}$ 
      (so we have ${\rb{S} = \rc{S} \!\. \setmin \:\! {\ri{S}} \.}$). 
   \end{enumerate} 
\end{definition}

\bigskip

\begin{proposition} \label{prop:ri-rc-rb} 
   Let $V \!\!$ be a topological real vector space. 
   
   \smallskip
   
   \begin{enumerate}[(1)]
      \item For any subsets $A \.$ and $B \.$ of $V \!\!$, we have the implication 
      ${A \inc B \imp \rc{A} \inc {\rc{B}} \.}$. 
      
      \medskip
      
      \item For any subset $A \.$ of $V \!\!$, we have 
      
      \smallskip
      
      \begin{enumerate}[(a)]
         \item ${\Aff{{\rc{A}} \.} = {\Aff{A}} \.}$, and 
         
         \smallskip
         
         \item $\ri{A} \neq \Oset \. 
         \iff \big( A \neq \Oset \. \ \ \mbox{and} \ \ \Aff{{\ri{A}} \.} = {\Aff{A}} \. \big) \.$. 
      \end{enumerate} 
      
      \medskip
      
      \item For any subset $A \.$ of $V \!\!$ and any affine subspace $W \!\!$ of $V \!\!$, we have 
      
      \smallskip
      
      \begin{enumerate}[(a)]
         \item ${\ri{A} \cap \Aff{A \cap W} \inc {\ri{A \cap W}} \.}$, and 
         
         \smallskip
         
         \item $\rb{A \cap W} \inc \rb{A} \cap W \!\.$. 
      \end{enumerate} 
   \end{enumerate} 
\end{proposition}

\medskip

\begin{proof}~\\ 
\textsf{Point~1.} 
Given subsets $A$ and $B$ of $V \!$ with $A \inc B$, we can write 
$$
\rc{A} \ = \ \clos{A} \cap \Aff{A} \ \inc \ \clos{B} \cap \Aff{B} \ = \ \rc{B}
$$ 
since we have ${\Aff{A} \inc {\Aff{B}} \;\!\!}$ and $\clos{A} \inc \clos{B}$. 

\medskip

\textsf{Point~2.a.} 
Using ${A \inc {\rc{A}} \.}$, we first get ${\Aff{A} \inc {\Aff{{\rc{A}} \.}} \.}$. 

\smallskip

Conversely, we have ${\rc{A} \inc {\Aff{A}} \;\!\!}$ 
by the very definition of ${{\rc{A}} \.}$, 
and hence one obtains the inclusion ${\Aff{{\rc{A}} \.} \inc \Aff{{\Aff{A}} \.} = {\Aff{A}} \.}$. 

\medskip

\textsf{Point~2.b.} 
If the open set ${{\ri{A}} \;\!\!}$ in ${{\Aff{A}} \;\!\!}$ is not empty, 
then we have ${\Aff{{\ri{A}} \.} = {\Aff{A}} \;\!\!}$ 
by Point~1.c in Remark~\ref{rem:0-tvs}, and $A$ is not empty since one has $\ri{A} \inc A$. 

\smallskip

Conversely, if we have $A \neq \Oset \.$ and ${\Aff{{\ri{A}} \.} = {\Aff{A}} \.}$, 
then we get $\Aff{{\ri{A}} \.} = \Aff{A} \neq \Oset \.$, and hence $\ri{A} \neq \Oset \.$ 
by using the obvious equality $\Aff{\Oset} = \Oset \.$. 

\medskip

\textsf{Point~3.a.} 
Since ${{\ri{A}} \;\!\!}$ is open in ${{\Aff{A}} \.}$, 
the intersection ${\ri{A} \cap {\Aff{A \cap W}} \;\!\!}$ 
is open in the subspace ${\Aff{A \cap W} \inc {\Aff{A}} \.}$. 

\smallskip

On the other hand, we have $\ri{A} \inc A$ and $\Aff{A \cap W} \inc \Aff{W} = W \!\!$, 
and hence the inclusion $\ri{A} \cap \Aff{A \cap W} \inc A \cap W \!$ holds. 

\smallskip

Therefore, we get ${\ri{A} \cap \Aff{A \cap W} \inc {\ri{A \cap W}} \;\!\!}$ 
since ${{\ri{A \cap W}} \;\!\!}$ is the largest open set in ${{\Aff{A \cap W}} \;\!\!}$ 
which is contained in $A \cap W \!\.$. 

\medskip

\textsf{Point~3.b.} 
We first have $\Aff{A \cap W} \inc \Aff{W} = W \!\!$, 
and hence ${{\rb{A \cap W}} \;\!\!}$ lies in $W \!\.$. 
On the other hand, combining Point~1 and Point 3.a yields 
\begin{eqnarray*} 
   \rb{A \cap W} 
   & = & \:\! 
   \rc{A \cap W} \!\. \setmin \:\! \ri{A \cap W} \\ 
   & \inc & 
   \rc{A \cap W} \!\. \setmin \:\! [\ri{A} \cap {\Aff{A \cap W}} \. ] \\ 
   & = & 
   [\rc{A \cap W} \!\. \setmin \:\! {\ri{A}} \. ] 
   \cup [\rc{A \cap W} \!\. \setmin \:\! {\Aff{A \cap W}} \. ] \\ 
   & = & 
   \rc{A \cap W} \!\. \setmin \:\! \ri{A} \\ 
   & & \mbox{(use ${\rc{A \cap W} \inc {\Aff{A \cap W}} \.}$)} \\ 
   & \inc & \rc{A} \!\. \setmin \:\! \ri{A} \ = \ \:\! {\rb{A}} \.~. 
\end{eqnarray*} 
\end{proof}

\bigskip

\begin{proposition} \label{prop:int-comp} 
   Let $X \!\.$ be a subset of a topological real vector space $V \!\!$ 
   and $A \.$ a subset of $X \. \cart \RR$ such that they satisfy $\Aff{A} = \Aff{X} \cart \RR$. 
   Then we have the following properties: 
   
   \smallskip
   
   \begin{enumerate}[(1)]
      \item The interior of $A \.$ in $X \. \cart \RR$ contains ${{\ri{A}} \.}$. 
      
      \medskip
      
      \item The closure of $A \.$ in $X \. \cart \RR$ lies in ${{\rc{A}} \.}$. 
   \end{enumerate} 
\end{proposition}

\medskip

\begin{proof}~\\ 
\textsf{Point~1.} 
Given ${\;\!\! (x_{0} , r_{0}) \in {\ri{A}} \.}$, 
there exist a neighborhood $U \!$ of $x_{0}$ in $V \!$ 
and a neighborhood $W \!$ of $r_{0}$ in $\RR$ 
such that we have ${\;\!\! [U \. \cap {\Aff{X}} \. ] \cart W \. \inc A}$, which implies 
$$
(U \. \cap X) \cart W \! 
\ = \ \;\!\. 
(X \. \cart \RR) \cap [U \. \cap {\Aff{X}} \. ] \cart W \. 
\ \inc \ 
(X \. \cart \RR) \cap A \ = \ A~.
$$ 
Then, since $U \. \cap X \.$ is a neighborhood of $x_{0}$ in $X \.$, 
we get that $\;\!\! (x_{0} , r_{0}) \;\!\!$ is in the interior of $A$ with respect to $X \. \cart \RR$. 

\medskip

\textsf{Point~2.} 
From ${X \. \cart \RR \inc \Aff{X} \cart \RR = {\Aff{A}} \.}$, 
we deduce that the closure $\clos{A}^{X \. \cart \RR} \!$ of $A$ in $X \. \cart \RR$ satisfies 
$$
\clos{A}^{X \. \cart \RR} \! 
\ = \ 
\clos{A} \cap (X \. \cart \RR) 
\ \inc \ 
\clos{A} \cap \Aff{A} \ = \ \:\! {\rc{A}} \.~.
$$ 
\end{proof}

\bigskip

Now, here is the definition of a strictly convex set, which is the key notion of the present work. 

\bigskip

\begin{definition} \label{def:strict-cvx} 
   A subset $C \.$ of a topological real vector space $V \!$ is said to be \emph{strictly convex} 
   if for any two distinct points ${x , y \in {\rc{C}} \;\!\!}$ 
   one has ${\;\!\! {]x , y[} \inc {\ri{C}} \.}$. 
\end{definition}

\bigskip

\begin{remark*} 
~
\begin{enumerate}[1)]
\item It is to be noticed that strict convexity is a topological property 
whereas convexity is a mere affine property. 

\smallskip

\item A strictly convex set is of course convex. 

\smallskip

\item According to the common geometric intuition, 
saying that a subset $C \.$ of $V \!$ is strictly convex means that $C \.$ is convex 
and that there is no non-trivial segment in the relative boundary of~$C \.$. 
Owing to Proposition~16 in~\cite[Chapitre~II, page~15]{BouEVT81}, 
this is an easy consequence of the very definition of strict convexity. 

\smallskip

\item This definition coincides with the usual one 
when $V \!$ is the canonical topological real vector space $\Rn{n}$ 
since in this case the closeness of ${{\Aff{C}} \;\!\!}$ in $V \!$ yields $\rc{C} = \:\! \clos{C} \.$. 
\end{enumerate} 
\end{remark*}

\bigskip

\begin{proposition} \label{prop:ri-strict-cvx} 
   For any strictly convex subset $C \!$ of a topological real vector space $V \!\!$, 
   we have the implication 
   
   \smallskip
   
   \begin{center} 
   $C \ \neq \ \Oset \. \qquad \imp \qquad  \ri{C} \ \neq \ \Oset \.$. 
   \end{center} 
\end{proposition}

\medskip

\begin{proof}~\\ 
There are two cases to be considered, 
depending on whether ${{\rc{C}} \;\!\!}$ is a single point or not. 

\smallskip

If we have $\rc{C} = \;\!\! \{ x \} \.$, then $C \.$ also reduces to $\. \{ x \} \.$ 
since we have ${\Oset \. \neq \:\! C \inc {\rc{C}} \.}$. 
Therefore, we obtain $\Aff{C} = \;\!\! \{ x \} \.$, 
and this implies $\ri{C} = \;\!\! \{ x \} \. \neq \Oset \.$. 

\smallskip

On the other hand, if we have ${x , y \in {\rc{C}} \;\!\!}$ with $x \neq y$, 
then the inclusion ${\;\!\! {]x , y[} \inc {\ri{C}} \;\!\!}$ holds 
by strict convexity of $C \.$, and hence one has $\ri{C} \neq \Oset \.$ 
since $\;\!\! {]x , y[} \;\!\!$ is not empty. 
\end{proof}

\bigskip

\begin{remark} 
When dealing with a single strictly convex subset $C \.$ 
of a general topological real vector space $V \!\!$, 
we will always assume in the hypotheses that $C \.$ has a non-empty interior in $V \!$ 
in order to insure $\Aff{C} = V \!\!$, and this makes sense by Proposition~\ref{prop:ri-strict-cvx} 
and Point~2.b in Proposition~\ref{prop:ri-rc-rb}. 
\end{remark}

\bigskip

We shall now prove that strict convexity is a two-dimensional (topological) notion 
whereas convexity is---by its very definition---a one-dimensional (affine) notion. 

\bigskip

\begin{proposition} 
   Given a topological real vector space $V \!\!$ with $\dim{V} \geq \:\! 2$, 
   any subset $C \!$ of $V \!\!$ whose interior is not empty satisfies the equivalence 
   
   \smallskip
   
   \begin{center} 
      $C \!$ is strictly convex 
      \quad $\iff$ \quad 
      $C \cap P \!$ is strictly convex for any affine plane $P \!$ in $V \!\!$~. 
   \end{center} 
\end{proposition}

\medskip

\begin{proof}~\\ 
\textasteriskcentered \ ($\imp$) 
Let ${x , y \in {\rc{C \cap P}} \;\!\!}$ with $x \neq y$. 

\smallskip

Then, we first have $x , y \in \clos{C} \.$ 
since the inclusion $\rc{C \cap P} \inc \:\! \clos{C} \.$ holds 
according to Point~1 in Proposition~\ref{prop:ri-rc-rb}, 
and this yields ${\;\!\! {]x , y[} \inc \:\! \intr{C} \.}$ by strict convexity of $C \.$. 

\smallskip

On the other hand, we have $x , y \in \. P \.$ 
from ${\rc{C \cap P} \inc \Aff{C \cap P} \inc \Aff{P} = P \!\.}$, 
and hence we get ${\;\!\! {]x , y[} \inc \:\! \intr{C} \cap P \.}$ by convexity of $P \!\.$. 

\smallskip

In particular, the open set ${\intr{C} \cap P \.}$ of $P \.$ is not empty, 
and this implies ${\Aff{\! \intr{C} \cap P \!} = P \.}$ by Point~3 in Remark~\ref{rem:0-tvs}, 
which yields ${P \inc {\Aff{C \cap P}} \;\!\!}$ since we have $\intr{C} \inc \:\! C \.$. 

\smallskip

Therefore, we get 
${\intr{C} \cap P \inc \:\! \intr{C} \cap \Aff{C \cap P} \inc {\ri{C \cap P}} \;\!\!}$ 
by Point~3.a in Proposition~\ref{prop:ri-rc-rb}, 
which gives ${\;\!\! {]x , y[} \inc {\ri{C \cap P}} \.}$. 

\smallskip

This proves that $C \cap P \.$ is strictly convex. 

\medskip

\textasteriskcentered \ ($\con$) 
Let $x , y \in \clos{C} \.$ with $x \neq y$. 

\smallskip

Since the dimension of ${V \! = {\Aff{C}} \;\!\!}$ is greater than one, 
$C \.$ does not lie in a line, 
and hence there exists a point $z \in \intr{C} \.$ such that $x , y , z$ are not collinear. 
Therefore, ${P \as {\Aff{\. \{ x , y , z \} \.}} \;\!\!}$ is an affine plane in $V \!\.$. 

\smallskip

Then, Proposition~16 in~\cite[Chapitre~II, page~15]{BouEVT81} implies that the open segments 
$\;\!\! {]z , x[} \;\!\!$ and $\;\!\! {]z , y[} \;\!\!$ are in $\intr{C} \cap P \!\.$, 
and hence in $C \cap P \!\.$. 
This insures that ${{\rc{]z , x[}} \;\!\!}$ and ${{\rc{]z , y[}} \;\!\!}$ 
are in ${{\rc{C \cap P}} \.}$. 

\smallskip

But we have ${x \in {\rc{]z , x[}} \;\!\!}$ since any neighborhood $U \!$ of $x$ in $V \!$ 
satisfies ${\;\!\! {]z , x[} \cap U \. = \Oset \.}$ by Point~1.a in Remark~\ref{rem:0-tvs}. 
And the same argument yields ${y \in {\rc{]z , y[}} \.}$. 

\smallskip

So, we have obtained ${x , y \in {\rc{C \cap P}} \.}$, 
and hence ${\;\!\! {]x , y[} \inc {\ri{C \cap P}} \;\!\!}$ since $C \cap P \.$ is strictly convex. 

\smallskip

Now, if we pick ${u \in {]x , y[} \;\!\!}$ and define $v \as u - z$, 
then there exists a number ${t > 0}$ that satisfies 
$w \as u + t v \in C \cap P \.$ since $u + \RR v$ is the straight line 
passing through ${u , z \in C \cap P \inc {\Aff{C \cap P}} \;\!\!}$ 
and since $C \cap P \.$ is a neighborhood of ${u \in {\ri{C \cap P}} \;\!\!}$ 
in ${{\Aff{C \cap P}} \.}$. 

\smallskip

Therefore, the segment $\;\!\! [z , w] \;\!\!$ lies in the convex set $C \cap P \!\.$, 
and hence in $C \.$, which yields 
$\;\!\! [z , u] \inc [z , w] \. \setmin \{ w \} \inc \:\! \intr{C} \.$ 
by Proposition~16 in \cite[Chapitre~II, page~15]{BouEVT81}. 

\smallskip

This proves that $C \.$ is strictly convex. 
\end{proof}

\bigskip

\begin{remark*} 
In case $V \!$ is one-dimensional but its topology is not Hausdorff, 
the strictly convex subsets of $V \!\!$, unlike those of $\RR$ (endowed with its usual topology), 
do \emph{not} coincide with its convex subsets. 
Indeed, if the topology of $V \!$ is for example trivial, 
then the only strictly convex subsets of $V \!$ are the empty set and $V \!$ itself. 
\end{remark*}

\bigskip

Finally, in order to be complete, let us recall the definition of a (strictly) convex function. 

\bigskip

\begin{definition} 
   Given a convex subset $C \.$ of a real vector space $V \!\!$, 
   a function ${\. f \. : C \to \RR}$ is said to be 
   
   \smallskip
   
   \begin{enumerate}[(1)]
      \item \emph{convex} if we have 
      $f \. ( \. ( \:\!\. 1 \;\!\! - t) x + t y) 
      \leq 
      ( \:\!\. 1 \;\!\! - t) \:\!\! f \. (x) \;\!\! + t \:\!\! f \. (y)$ 
      for any points $x , y \in C \.$ and any number $t \in (0 , \;\!\! 1 \:\!\. ) \.$, and 
      
      \smallskip
      
      \item \emph{strictly convex} if we have 
      $f \. ( \. ( \:\!\. 1 \;\!\! - t) x + t y) 
      < 
      ( \:\!\. 1 \;\!\! - t) \:\!\! f \. (x) \;\!\! + t \:\!\! f \. (y)$ 
      for any distinct points $x , y \in C \.$ and any number $t \in (0 , \;\!\! 1 \:\!\. ) \.$. 
   \end{enumerate} 
\end{definition}

\bigskip

\begin{remark} \label{rem:sc} 
~
\begin{enumerate}[1)]
\item It is to be noticed that both convexity and strict convexity of functions 
are mere affine notions. 

\medskip

\item A strictly convex function is of course convex. 

\medskip

\item Given a convex subset $C \.$ of a real vector space $V \!\!$, 
a function $\. f \. : C \to \RR$ is convex if and only if one has 
\[
f \!\. \left( \! \sum_{i = \. 1}^{n} \l_{i} x_{i} \!\! \right) \!\!\: 
\ \leq \ \!\. 
\sum_{i = \. 1}^{n} \l_{i} f \. (x_{i}) 
\qquad \mbox{(Jensen's inequality)}
\] 
for any integer $n \geq \;\!\! 1$, any points $\llist{x}{1}{n \!} \in C \.$, 
and any numbers $\llist{\l}{1}{n \!} \in [0 , {+\infty}) \;\!\!$ 
which satisfy $\disp \,\!\! \sum_{i = \. 1}^{n} \l_{i} = \;\!\! 1$. 

\smallskip

This is obtained by induction on $n \geq \;\!\! 1$. 

\medskip

\item Given a convex subset $C \.$ of a real vector space $V \!\!$, 
a convex function $\. f \. : C \to \RR$, 
a subset $A \inc \:\! C \.$ and a real number $M > 0$, we have the following equivalence: 
$$
\big( \forall x \in \. A, \ \ \. f \. (x) \ \leq \ M \big) 
\qquad \iff \qquad 
\big( \forall x \in \Conv{A} \! , \ \ \. f \. (x) \ \leq \ M \big) \.~,
$$ 
where ${{\Conv{A}} \;\!\!}$ stands for the convex hull of $A$, 
\ie\!\!, the smallest convex subset of $V \!$ which contains $A$. 

\smallskip

Indeed, it is a well-known fact that each ${x \in {\Conv{A}} \;\!\!}$ 
writes ${\disp x = \!\!\. \sum_{i = \. 1}^{n} \l_{i} x_{i}}$ 
for some integer $n \geq \;\!\! 1$, some points ${\llist{x}{1}{n \!} \in \. A}$ 
and some numbers ${\llist{\l}{1}{n \!} \in [0 , {+\infty}) \;\!\!}$ 
which satisfy ${\disp \,\!\! \sum_{i = \. 1}^{n} \l_{i} = \;\!\! 1}$.
Therefore, this implies ${\disp \. f \. (x) \leq \!\. \sum_{i = \. 1}^{n} \l_{i} f \. (x_{i}) 
\leq \!\. \sum_{i = \. 1}^{n} \l_{i} M = M}$ by using Point~3 above. 
\end{enumerate} 
\end{remark}

\bigskip
\bigskip


\section{Proofs} \label{sec:proofs} 

This section is devoted to the proofs of the Main~Theorem and of all its consequences 
that we mentionned in Section~\ref{sec:main-thm}. 

\bigskip

Let us first begin with the following affine property. 

\bigskip

\begin{lemma} \label{lem:rb-line-epi} 
   For any function $\. f \. : S \to \RR$ defined on a subset $S \!$ of a real vector space, 
   we have $\Aff{{\Epi{f}} \.} = \Aff{S} \cart \RR$. 
\end{lemma}

\medskip

\begin{proof}~\\ 
We obviously have $\Epi{f} \inc S \cart \RR \inc \Aff{S} \cart \RR$, 
and hence $\Aff{{\Epi{f}} \.} \inc \Aff{S} \cart \RR$. 

\smallskip

Conversely, given $x \in S$, we have 
${\. \{ x \} \. \cart [f \. (x) , \. f \. (x) \;\!\! + \;\!\! 1 \:\!\. ] \inc {\Epi{f}} \;\!\!}$ 
by the very definition of ${{\Epi{f}} \.}$, which gives ${\. \{ x \} \. \cart \RR 
= \Aff{\. \{ x \} \. \cart [f \. (x) , \. f \. (x) \;\!\! + \;\!\! 1 \:\!\. ]} 
\inc {\Aff{{\Epi{f}} \.}} \;\!\!}$ by Proposition~\ref{prop:aff-cart}. 
Therefore, we get ${S \cart \RR \inc {\Aff{{\Epi{f}} \.}} \.}$, 
and this yields ${\Aff{S} \cart \RR = \Aff{S \cart \RR} \inc {\Aff{{\Epi{f}} \.}} \;\!\!}$ 
by Proposition~\ref{prop:aff-cart} once again. 
\end{proof}

\bigskip

Then, let us establish two technical but useful topological properties. 

\bigskip

\begin{lemma} \label{lem:epi-closed} 
   Given a subset $S \!$ of a topological real vector space $V \!\!$ 
   and a function $\. f \. : S \to \RR$, 
   we have the implication 
   $$
   \left( \!\!\! 
   \begin{array}{c} 
      \vspace{5pt} 
      S \! \ \mbox{is open in} \ {\Aff{S}} \.~, \\ 
      \vspace{5pt} 
      f \! \ \mbox{is lower semi-continuous}~, \\ 
      \forall x_{0} \in \rb{S} \inc \clos{S}, \ \ \. 
      f \. (x) \ \to \ \:\! {+\infty} \quad \mbox{as} \quad x \ \to \ x_{0} 
   \end{array} \!\! 
   \right) 
   \: \imp \: 
   {\Epi{f}} \;\!\! \ \mbox{is closed in} \ \Aff{S} \cart \RR~.
   $$ 
\end{lemma}

\smallskip

\begin{proof}~\\ 
Assume that all the hypotheses are satisfied, and let ${\;\!\! (a , \a) \in {\rc{{\Epi{f}} \.}} \.}$. 

\smallskip

First of all, notice that $a$ is in ${{\rc{S}} \;\!\!}$ 
since the projection of $V \. \cart \RR$ onto $V \!$ is continuous and since we have 
$\Aff{{\Epi{f}} \.} = \Aff{S} \cart \RR$ by Lemma~\ref{lem:rb-line-epi}. 

\smallskip

If we had ${a \not \in S = {\ri{S}} \;\!\!}$ 
(remember that $S$ is open in ${{\Aff{S}} \.}$), then we would get ${a \in {\rb{S}} \.}$, 
and hence $\. f \. (x) \to \:\! {+\infty}$ as $x \to x_{0} \as a$, 
which yields ${\;\!\! (a , \a) \not \in {\rc{{\Epi{f}} \.}} \;\!\!}$ 
by Proposition~\ref{prop:epi-infinity}, a contradiction. 

\smallskip

Therefore, the point $a$ is necessarily in $S \.$. 

\smallskip

Now, assume that we have $\. f \. (a) > \:\! \a$, 
and let $\e \:\!\. \as \;\!\! (f \. (a) \;\!\! - \a) \! / \. 2 > 0$. 

\smallskip

Since $\. f \.$ is lower semi-continuous at $a$, 
there exists a neighborhood $U \!$ of $a$ in $S$ that satisfies the inclusion 
$\. f \. (U) \inc [f \. (a) \;\!\! - \e \, , \, {+\infty}) = \;\!\. [\a + \e \, , \, {+\infty}) \.$. 

\smallskip

But $S$ is a neighborhood of $a$ in ${{\Aff{S}} \.}$, and hence so is $U \!$. 

\smallskip

Thus, combining ${\;\!\! (a , \a) \in {\rc{{\Epi{f}} \.}} \;\!\!}$ 
and $\Aff{{\Epi{f}} \.} = \Aff{S} \cart \RR$, 
we can find $x \in U \!$ and $\l \in (\a - \e \, , \, \a + \e) \;\!\!$ with $\. f \. (x) \leq \:\! \l$, 
which yields $\. f \. (x) \in ( \. {-\infty} \, , \, \a + \e) \.$, 
contradicting the inclusion above. 

\smallskip

So, we necessarily have $\. f \. (a) \leq \:\! \a$, 
or equivalently ${\;\!\! (a , \a) \in {\Epi{f}} \.}$. 

\smallskip

Conclusion: we proved ${\rc{{\Epi{f}} \.} \inc {\Epi{f}} \.}$, 
which means that ${{\Epi{f}} \;\!\!}$ is closed in $\Aff{{\Epi{f}} \.} = \Aff{S} \cart \RR$. 
\end{proof}

\bigskip

\begin{lemma} \label{lem:epi-strict} 
   For any function $\. f \. : C \to \RR$ defined on a convex subset $C \!$ 
   of a topological real vector space $V \!\!$, we have (1)$\imp$(2)$\imp$(3)$\imp$(4) with 
   
   \smallskip
   
   \begin{enumerate}[(1)]
      \item ${{\Epi{f}} \;\!\!}$ is strictly convex, 
      
      \medskip
      
      \item $\forall s > 0, \ \tau_{\! s}({\rc{{\Epi{f}} \.}} \. ) 
      \inc \ri{{\Epi{f}} \.} \inc \ri{C} \cart \RR$, 
      
      \medskip
      
      \item $\rc{{\Epi{f}} \.} \inc \ri{C} \cart \RR$, and 
      
      \medskip
      
      \item $C \! \ \mbox{is open in} \ {\Aff{C}} \.$, 
   \end{enumerate} 
   
   \smallskip
   
   where for each $s \in \RR$ the map $\tau_{\! s} : V \. \cart \RR \to V \. \cart \RR$ 
   is defined by $\tau_{\! s}(x , r) \. \as \;\!\! (x , r + s) \.$. 
\end{lemma}

\medskip

\begin{proof}~\\ 
\textsf{Point~1 $\imp$ Point~2.} 
Since for each ${s \in \RR}$ the map $\tau_{\! s}$ is an affine homeomorphism, we have the equality 
${\tau_{\! s}({\rc{{\Epi{f}} \.}} \. ) = \:\! {\rc{\. \tau_{\! s}({\Epi{f}} \. ) \.}} \.}$. 

\smallskip

Then, for any $s \geq 0$, we obtain 
${\tau_{\! s}({\rc{{\Epi{f}} \.}} \. ) \inc {\rc{{\Epi{f}} \.}} \;\!\!}$ 
since one has ${\tau_{\! s}({\Epi{f}} \. ) \inc {\Epi{f}} \.}$. 

\smallskip

Hence, given ${s > 0}$ and ${\;\!\! (x , r) \in {\rc{{\Epi{f}} \.}} \.}$, 
we can write ${\;\!\! (x , r + \;\!\! 2 s) = \:\! \tau_{\. 2 s}(x , r) \in {\rc{{\Epi{f}} \.}} \.}$, 
which implies that the midpoint ${\tau_{\! s}(x , r) = \;\!\. (x , r + s) \;\!\!}$ 
is in ${{\ri{{\Epi{f}} \.}} \;\!\!}$ since ${{\Epi{f}} \;\!\!}$ is strictly convex. 
This proves the first inclusion. 

\smallskip

The second inclusion is straightforward since we have ${\Aff{{\Epi{f}} \.} = \Aff{C} \cart \RR}$ 
by Lemma~\ref{lem:rb-line-epi}. 

\medskip

\textsf{Point~2 $\imp$ Point~3.} 
Fixing an arbitrary real number $s > 0$, Point~2 implies 
$\tau_{\! s}({\rc{{\Epi{f}} \.}} \. ) \inc \ri{C} \cart \RR$, 
and hence $\rc{{\Epi{f}} \.} \inc \tau_{\! {-s}}(\ri{C} \cart \RR) = \:\! \ri{C} \cart \RR$. 

\medskip

\textsf{Point~3 $\imp$ Point~4.} 
Since we have ${\Epi{f} \inc {\rc{{\Epi{f}} \.}} \.}$, 
Point~3 implies ${\Epi{f} \inc \ri{C} \cart \RR}$, 
which gives ${C \inc {\ri{C}} \;\!\!}$ by applying the projection of $V \. \cart \RR$ onto $V \!\.$. 
\end{proof}

\bigskip

Combining Lemma~\ref{lem:epi-closed} and Lemma~\ref{lem:epi-strict} 
with all the properties established in Section~\ref{sec:preliminaries}, 
we are now able to prove the Main~Theorem. 

\bigskip

\begin{proof}[Proof of the Main~Theorem]~\\ 
We may assume that $C \.$ is not empty since this equivalence is obviously true otherwise. 

\medskip

\textasteriskcentered \ ($\imp$) 
Let ${\;\!\! (x , r) , (y , s) \in {\rc{{\Epi{f}} \.}} \;\!\!}$ with $\;\!\! (x , r) \neq (y , s) \.$, 
fix $t \in (0 , \;\!\! 1 \:\!\. ) \.$, and define 
$$
(a , \a) \. \as \;\!\! ( \:\!\. 1 \;\!\! - t) \. (x , r) \;\!\! + t (y , s) 
\ = \ \;\!\. 
( \. ( \:\!\. 1 \;\!\! - t) x + t y \ , \ ( \:\!\. 1 \;\!\! - t) r + t s) \ \in \ V \. \cart \RR~.
$$ 


By Lemma~\ref{lem:epi-closed}, we already have ${\;\!\! (x , r) , (y , s) \in {\Epi{f}} \.}$. 
Then, since $C \.$ and $\. f \.$ are both convex, 
${{\Epi{f}} \;\!\!}$ is convex by Proposition~\ref{prop:convex-epi}, 
which implies ${\;\!\! (a , \a) \in {\Epi{f}} \.}$. 

\smallskip

There are now two cases to be considered. 

\smallskip

\textbullet \ Case $x = y$ and $r < s$. 

Here we have $a = x = y$, which yields 
\[
f \. (a) \ = \ \:\! \:\!\! f \. (x) \ \leq \ \:\! r \:\! 
\ = \ \;\!\. 
( \:\!\. 1 \;\!\! - t) r + t r \:\! 
\ < \ ( \:\!\. 1 \;\!\! - t) r + t s \ = \ \a.
\] 

\textbullet \ Case $x \neq y$. 

By strict convexity of $\. f \!$, we have 
$\. f \. (a) < ( \:\!\. 1 \;\!\! - t) \:\!\! f \. (x) \;\!\! + t \:\!\! f \. (y) 
\leq ( \:\!\. 1 \;\!\! - t) r + t s = \a$. 

\smallskip

In both cases, we get ${\;\!\! (a , \a) \in {\ri{{\Epi{f}} \.}} \;\!\!}$ 
by Point~2 in Proposition~\ref{prop:epi-usc} since $\. f \.$ is upper semi-continuous at $a$. 

\smallskip

This proves that ${{\Epi{f}} \;\!\!}$ is strictly convex. 

\medskip

\textasteriskcentered \ ($\con$) 
First of all, $C \.$ is convex by Proposition~\ref{prop:convex-epi}. 
Moreover, $C \.$ is open in ${{\Aff{C}} \;\!\!}$ 
by the third implication in Lemma~\ref{lem:epi-strict}. 

\smallskip

On the other hand, given $x , y \in C \.$ with $x \neq y$ 
and $t \in (0 , \;\!\! 1 \:\!\. ) \.$, 
the points $\;\!\! (x , \. f \. (x) \. ) \;\!\!$ and $\;\!\! (y , \. f \. (y) \. ) \;\!\!$ 
are in ${\Epi{f} \inc {\rc{{\Epi{f}} \.}} \.}$, which yields 
\[
(a , \a) \. \as \;\!\! ( \. ( \:\!\. 1 \;\!\! - t) x + t y \ , \ 
( \:\!\. 1 \;\!\! - t) \:\!\! f \. (x) \;\!\! + t \:\!\! f \. (y) \. ) 
\ \in \ 
\ri{{\Epi{f}} \.}
\] 
since ${{\Epi{f}} \;\!\!}$ is strictly convex. 

\smallskip

Therefore, we get 
${\. f \. ( \. ( \:\!\. 1 \;\!\! - t) x + t y) = \:\! \:\!\! f \. (a) < \:\! 
\a \:\! = \;\!\. ( \:\!\. 1 \;\!\! - t) \:\!\! f \. (x) \;\!\! + t \:\!\! f \. (y) \;\!\!}$ 
by Point~1 in Proposition~\ref{prop:int-comp} with ${X \. \as C \.}$ and ${A \as {\Epi{f}} \;\!\!}$ 
by using ${\Aff{{\Epi{f}} \.} = \Aff{C} \cart \RR}$ (see Lemma~\ref{lem:rb-line-epi}) and by Point~1 
in Proposition~\ref{prop:epi-intr+clos} with ${X \. \as C \.}$. 
This proves that $\. f \.$ is strictly convex. 

\smallskip

Now, since ${{\Epi{f}} \;\!\!}$ is strictly convex and non-empty, 
we have $\ri{{\Epi{f}} \.} \neq \Oset \.$ by Proposition~\ref{prop:ri-strict-cvx}. 
On the other hand, since we have ${\Aff{{\Epi{f}} \.} = \Aff{C} \cart \RR}$ 
by Lemma~\ref{lem:rb-line-epi}, we can apply Proposition \ref{prop:loc-bounded} 
with ${X \. \as {\Aff{C}} \;\!\!}$ and $S \as C \.$, and then obtain that 
$\. f \.$ is bounded from above on a subset of $C \.$ which is non-empty and open in ${{\Aff{C}} \.}$. 
But this implies that $\. f \.$ is continuous 
by Proposition~21 in~\cite[Chapitre~II, page~20]{BouEVT81}. 

\smallskip

Finally, given ${x_{0} \in {\rb{C}} \.}$, we have 
$\;\!\! ( \. \{ x_{0} \} \. \cart \RR) \cap \rc{{\Epi{f}} \.} = \Oset \.$ 
by using the second implication in Lemma~\ref{lem:epi-strict}, 
and hence $\. f \. (x) \to \:\! {+\infty}$ as $x \to x_{0}$ by Proposition~\ref{prop:epi-infinity}. 
\end{proof}

%

\bigskip

\begin{proof}[Proof of Proposition~\ref{prop:sc-intr-epi}]~\\ 

\vspace{-18pt}

\textasteriskcentered \ ($\imp$) 
Since ${{\Epi{f}} \;\!\!}$ is not empty, the same holds for ${\intr{\wideparen{\Epi{f}}} \;\!\!}$ 
by Proposition~\ref{prop:ri-strict-cvx}. 

\medskip

\textasteriskcentered \ ($\con$) 
By Proposition~\ref{prop:loc-bounded} with ${X \. \as V \!}$ and ${S \as V \!\!}$, 
we get that $\. f \.$ is bounded from above on a non-empty open set in $V \!\!$, 
and hence it is continuous by Proposition~21 in~\cite[Chapitre~II, page~20]{BouEVT81}. 
Therefore, the Main~Theorem with $C \as V \!$ implies that ${{\Epi{f}} \;\!\!}$ is strictly convex. 
\end{proof}

\bigskip

\begin{proof}[Proof of Proposition~\ref{prop:epi-rest}]~\\ 
\textsf{Point~1 $\imp$ Point~2.} 
Given an affine subspace $G \.$ of $V \!\!$, the inclusion 
\[
{\Epi{\rest{f}{C \cap G \.}}} \;\!\! \ = \ \ \!\! \Epi{f} \cap (G \cart \RR) \ \inc \ \Epi{f}
\] 
yields 
\begin{eqnarray*} 
   \rc{{\Epi{\rest{f}{C \cap G \.}}} \:\!\!} 
   & = & \. 
   \clos{\, \Epi{\rest{f}{C \cap G \.}}} \cap (\Aff{C \cap G} \cart \RR) \\ 
   & \inc & \. 
   \clos{\, \Epi{\rest{f}{C \cap G \. }}} \cap (\Aff{C} \cart \RR) \\ 
   & \inc & \. 
   \clos{\, \Epi{f}} \cap (\Aff{C} \cart \RR) 
   \ = \ \:\! 
   {\rc{{\Epi{f}} \.}} \.~. 
\end{eqnarray*} 

\smallskip

Therefore, any two distinct points 
${\;\!\! (x , r) , (y , s) \in {\rc{{\Epi{\rest{f}{C \cap G \.}}} \:\!\!}} \.}$ 
are in ${\rc{{\Epi{f}} \.} \.}$, 
which implies that each ${\;\!\! (z , t) \in {](x , r) \, , \, (y , s)[} \;\!\!}$ 
is in ${{\ri{{\Epi{f}} \.}} \;\!\!}$ since ${{\Epi{f}} \;\!\!}$ is strictly convex. 

\smallskip

So, there exist $\e > 0$ and a neighborhood $\Om$ of $z$ in $V \!$ such that 
the set $\;\!\! (\Om \cap {\Aff{C}} \. ) \cart (t - \e \, , \, t + \e) \;\!\!$ 
is included in ${{\Epi{f}} \.}$, from which we get 
\begin{eqnarray*} 
   [\Om \cap {\Aff{C \cap G}} \. ] \cart (t - \e \, , \, t + \e) 
   & \inc & 
   (\Om \cap \Aff{C} \cap G) \cart (t - \e \, , \, t + \e) \\ 
   & = & 
   [ \. (\Om \cap {\Aff{C}} \. ) \cart (t - \e \, , \, t + \e) \. ] \cap (G \cart \RR) \\ 
   & \inc & 
   \Epi{f} \cap (G \cart \RR) 
   \ = \ 
   \Epi{\rest{f}{C \cap G \.}} 
\end{eqnarray*} 
since one has $\Aff{C \cap G} \inc \Aff{C} \cap G \.$ 
and ${\Epi{f} \cap (G \cart \RR) = {\Epi{\rest{f}{C \cap G \.}}} \.}$, 
proving that $\;\!\! (z , t) \;\!\!$ is in ${{\rc{{\Epi{\rest{f}{C \cap G \.}}} \:\!\!}} \.}$. 

\smallskip

This shows that ${{\Epi{\rest{f}{C \cap G \.}}} \.}$ is strictly convex. 

\medskip

\textsf{Point~2 $\imp$ Point~1.} 
This is clear. 

\medskip

\textsf{Point~2 $\imp$ Point~3.} 
This is straightforward. 
\end{proof}

\bigskip

Before proving Proposition~\ref{prop:epi-finite-dim}, we need two key lemmas. 

\bigskip

\begin{lemma} \label{lem:rb-line} 
   Given a convex subset $C \!$ of a topological real vector space $V \!\!$ 
   and a straight line $L \.$ in $V \!\!$, we have the implication 
   $$
   \left( 
   \begin{array}{c} 
      \vspace{5pt} 
      C \! \ \mbox{is open in} \ {\Aff{C}} \.~, \\ 
      \vspace{5pt} 
      C \cap L \ \neq \ \Oset \.~, \\ 
      \mbox{the subspace topology on} \ L \. \ \mbox{is Hausdorff} 
   \end{array} 
   \right) \! 
   \qquad \imp \qquad 
   \rb{C \cap L} \ = \ \:\! \rb{C} \cap L~.
   $$ 
\end{lemma}

\smallskip

\begin{proof}~\\ 
Assume that all the hypotheses are satisfied. 

\medskip

\textasteriskcentered \ 
Using Point~3.b in Proposition~\ref{prop:ri-rc-rb} with ${A \as C \.}$ and ${W \. \as L}$, 
we immediately have the inclusion $\rb{C \cap L} \inc \rb{C} \cap L$. 

\medskip

\textasteriskcentered \ 
Now, fix a point ${b \in C \cap L \inc \:\! C = \:\! {\ri{C}} \;\!\!}$ 
and let $a$ be an arbitrary point in $\rb{C} \cap L$. 

\smallskip

Since we have ${a \in {\rc{C}} \.}$, Proposition~16 in~\cite[Chapitre~II, page~15]{BouEVT81} 
implies $\;\!\! {]a , b[} \inc \ri{C} = \:\! C \.$, which yields $\;\!\! {]a , b[} \inc \:\! C \cap L$ 
since $a$ and $b$ lie in the convex set $L$. 

\smallskip

The points $a$ and $b$ being distinct, we obtain 
$L = \Aff{]a , b[} \inc \Aff{C \cap L} \inc \Aff{L} = L$, and hence $\Aff{C \cap L} = L$. 

\smallskip

Therefore, since we have ${L = \Aff{C \cap L} \inc {\Aff{C}} \;\!\!}$ 
and since $C \.$ is open in ${{\Aff{C}} \.}$, we then deduce that $C \cap L$ is open in $L$, 
which is equivalent to saying that $C \cap L$ is open in ${{\Aff{C \cap L}} \;\!\!}$ 
by using again $\Aff{C \cap L} = L$. 

\smallskip

Conclusion: we get $\rb{C \cap L} = \:\! {\rc{C \cap L}} \. \setmin (C \cap L) \.$. 

\smallskip

On the other hand, the inclusion ${\;\!\! {]a , b[} \inc \:\! C \cap L}$ we established above 
yields ${\rc{]a , b[} \inc {\rc{C \cap L}} \;\!\!}$ by Point~1 in Proposition~\ref{prop:ri-rc-rb}, 
which writes ${\;\!\! [a , b] \inc {\rc{C \cap L}} \;\!\!}$ since we have $\Aff{]a , b[} = L$ 
together with ${\!\. \clos{\, ]0 , \;\!\! 1 \:\!\. [} = \;\!\. [0 , \;\!\! 1 \:\!\. ] \;\!\!}$ 
and since the map $\c : \RR \to L$ defined by ${\c(t) \. \as a + t (b - a) \;\!\!}$ is a homeomorphism 
as a consequence of Theorem~2 in~\cite[Chapitre~I, page~14]{BouEVT81} 
and the fact that $L$ is a finite-dimensional affine space whose topology is Hausdorff. 

\smallskip

So, we obtain in particular ${a \in {\rc{C \cap L}} \.}$, 
and hence $a \in \rb{C \cap L} = {\rc{C \cap L}} \. \setmin (C \cap L) \;\!\!$ 
since one has $a \in V \setmin C \inc V \setmin (C \cap L) \.$. 
\end{proof}

\bigskip

\begin{lemma} \label{lem:line-infinity} 
   Given a convex subset $C \!$ of $\Rn{n} \.$, a convex function $\. f \. : C \to \RR$ 
   and a point $a \in \rb{C} \inc \:\! \clos{C} \!$, we have the equivalence 
   $$
   \big( f \. (x) \ \to \ \:\! {+\infty} \quad \mbox{as} \quad x \ \to \ a \big) 
   \ \iff \ \;\!\! 
   \left( 
   \begin{array}{c} 
      \vspace{5pt} 
      \mbox{for any straight line} \ L \. \ \mbox{in} \ \Rn{n} \. \ \mbox{passing through} \ a, \\ 
      f \. (x) \ \to \ \:\! {+\infty} \quad \mbox{as} \quad x \ \to \ a 
      \ \mbox{with} \ x \in C \cap L 
   \end{array} 
   \right) \!\!\!~.
   $$ 
\end{lemma}

\smallskip

\begin{proof}~\\ 
\textasteriskcentered \ ($\imp$) 
This implication is obvious. 

\medskip

\textasteriskcentered \ ($\con$) 
Assume that we have $\. f \. (x) \ \not \!\! \to \:\! {+\infty}$ as $x \to a$. 

\smallskip

Since we can write ${\rc{C} = \rc{{\ri{C}} \.} \;\!\!}$ 
by Corollary~1 in~\cite[Chapitre~II, bottom of page~15]{BouEVT81} 
and since we have ${a \in {\rc{C}} \.}$, this means that there exist 
a sequence ${\;\!\! \seqN{x}{k}}$ in ${{\ri{C}} \;\!\!}$ that converges to $a$ 
and a number $M > 0$ that satisfies $\. f \. (x_{k}) \leq \:\! M$ for any $k \in \NN$. 

\smallskip

Therefore, if we consider the set ${X \!\. \as \:\!\! \{x_{k} \st k \in \NN\} \.}$, 
one obtains $\. f \. (x) \leq \:\! M$ for any ${x \in {\Conv{X}} \;\!\!}$ 
according to Point~4 in Remark~\ref{rem:sc}. 

\smallskip

Now, noticing that $a$ lies in $\clos{X} \.$, 
we get ${a \in \rc{{\Conv{X}} \.} = \:\!\! \clos{\, \Conv{X}} \cap {\Aff{X}} \;\!\!}$ 
since we have $\clos{X} \inc \. \clos{\, \Conv{X}} \;\!\!$ 
and since ${{\Aff{X}} \;\!\!}$ is closed in $\Rn{n}$. 

\smallskip

On the other hand, since ${{\Conv{X}} \;\!\!}$ is not empty, 
the same holds for ${{\ri{{\Conv{X}} \.}} \.}$, 
which insures the existence of a point ${b \in {\ri{{\Conv{X}} \.}} \.}$. 

\smallskip

Proposition~16 in~\cite[Chapitre~II, page~15]{BouEVT81} 
then implies ${\;\!\! {]a , b[} \inc {\ri{{\Conv{X}} \.}} \.}$. 

\smallskip

Moreover, since ${{\ri{C}} \;\!\!}$ is convex 
by Corollary~1 in~\cite[Chapitre~II, bottom of page~15]{BouEVT81}, 
we have ${\ri{{\Conv{X}} \.} \inc \Conv{X} \inc \Conv{{\ri{C}} \.} = {\ri{C}} \.}$, 
and hence $a$ does not belong to ${{\ri{{\Conv{X}} \.}} \.}$ since we have ${a \not \in {\ri{C}} \.}$, 
which yields $b \neq a$. 

\smallskip

Finally, if $L$ denotes the straight line in $\Rn{n}$ passing through $a$ and $b$, 
we do not have ${\. f \. (x) \to \:\! {+\infty}}$ as $x \to a$ with $x \in C \cap L$ 
since the inequality $\. f \. (x) \leq \:\! M$ holds for any 
$x \in {]a , b[} \inc \ri{{\Conv{X}} \.} \cap L \inc \ri{C} \cap L \inc \:\! C \cap L$. 
\end{proof}

\bigskip

\begin{proof}[Proof of Proposition~\ref{prop:epi-finite-dim}]~\\ 
Since the equivalence is obvious when $C \.$ is empty or reduced to a single point, 
we may assume that $C \.$ has at least two distinct points. 

\medskip

\textsf{Point~1 $\imp$ Point~2.} 
This implication has already been proved in Proposition~\ref{prop:epi-rest}. 

\medskip

\textsf{Point~2 $\imp$ Point~1.} \\ 
\textasteriskcentered \ 
First of all, the intersection of $C \.$ with any straight line $L$ in $\Rn{n}$ 
is convex by applying the converse part of the Main~Theorem to the function $\. \rest{f}{C \cap L}$. 
Hence, $C \.$ is convex. 

\smallskip

\textasteriskcentered \ 
Let us then show that $C \.$ is open in ${{\Aff{C}} \.}$. 

\smallskip

Fix $x_{0} \in C \.$, 
and consider the subset ${G \. \as \:\!\! \{ x - x_{0} \st x \in C \} \.}$ of $V \!\.$. 

\smallskip

Then the vector subspace ${W \!\. \as \:\!\! \{ v - x_{0} \st v \in {\Aff{C}} \. \} \.}$ 
of $V \!$ is generated by $G \. $ since $C \.$ is a generating set 
of the affine space ${{\Aff{C}} \.}$. 

\smallskip

Hence, there exists a subset $B$ of $G \.$ which is a basis of $W \!\.$. 

\smallskip

Denoting by ${d \in \. \{ \:\!\. 1 \:\!\. , \ldots , n \} \.}$ 
the dimension of the affine subspace ${{\Aff{C}} \;\!\!}$ of $\Rn{n}$, we have $\dim{W} = \:\! d$, 
and hence there are vectors $\llist{x}{1}{d} \in C \.$ such that we can write 
${B = \;\!\! \{ x_{i} \. - x_{0} \st \:\!\. 1 \;\!\! \leq i \leq d \} \.}$. 

\smallskip

Now, for each ${i \in \. \{ \:\!\. 1 \:\!\. , \ldots , d \} \.}$, 
if $L_{i \.}$ denotes the straight line in $\Rn{n}$ passing through $x_{0}$ and $x_{i}$, 
the strict convexity of ${{\Epi{\rest{f}{C \cap L_{i} \.}}} \.}$ implies that $C \cap L_{i \.}$ 
is open in ${\Aff{C \cap L_{i}} = L_{i} = {\Aff{\. \{ x_{0} , x_{i} \} \.}} \;\!\!}$ 
by Lemma~\ref{lem:epi-strict}, which insures the existence of a number $r_{i} \. > 0$ such that 
the vector $v_{i} \. \as r_{i} (x_{i} \. - x_{0}) \;\!\!$ satisfies 
\[
\{ x_{0} - v_{i} \, , \, x_{0} + v_{i} \} \inc \:\! C \cap L_{i} \inc \:\! C \.~.
\] 


Therefore, the convex hull $U \!$ 
of ${\disp \bigcup_{i = \. 1}^{d} \{ x_{0} - v_{i} \, , \, x_{0} + v_{i} \} \.}$ 
lies in the convex set $C \.$. 

\smallskip

Since the map $\. f \. : \Rn{d} \to \Rn{n}$ defined by 
${\disp \. f \. \lvec{\l}{1}{d} \. \as \!\. \sum_{i = \. 1}^{d} \l_{i} v_{i \.}}$ is linear 
and sends the canonical ordered basis $\lvec{e}{1}{d}$ of $\Rn{d}$ 
to the family $\lvec{v}{1}{d}$ of $\Rn{n}$, 
we can write ${U \:\!\! = \:\! x_{0} + \:\!\! f \. (\Om) \.}$, where 
${\Om \. \as \:\!\! \{ \lvec{\l}{1}{d} \in \Rn{d} 
\st \! |\l_{1}| \;\!\! + \cdots + \;\!\! |\l_{d}| \leq \:\! \;\!\! 1 \:\!\. \} \.}$ 
is the convex hull of $\disp \bigcup_{i = \. 1}^{d} \{ \. {-e_{i}} \, , \, e_{i} \} \.$. 

\smallskip

But $\lvec{v}{1}{d}$ is a basis of $W \!$ since $B$ is, 
which implies that $\. f \.$ satisfies ${\Im{f} = W \!}$ and is a linear isomorphism onto its image. 
Therefore, $x_{0} + \:\!\! f \.$ is a homeomorphism from $\Rn{d}$ (endowed with its usual topology) 
onto ${x_{0} + W \:\!\! = {\Aff{C}} \.}$. 

\smallskip

As a consequence, we then get that $U \!$ is a neighbourhood of $x_0$ in ${{\Aff{C}} \;\!\!}$ 
since $\Om$ is a neighborhood of the origin in $\Rn{d}$, 
and hence $C \.$ is itself a neighborhood of $x_0$ in ${{\Aff{C}} \.}$. 

\smallskip

\textasteriskcentered \ 
On the other hand, $\. f \.$ is strictly convex since its restriction to any straight line 
in $\Rn{n}$ is strictly convex and since strict convexity is an affine notion. 

\smallskip

\textasteriskcentered \ 
Moreover, the convexity of $\. f \.$ on the open convex subset $C \.$ 
of the finite-dimensional affine space ${{\Aff{C}} \;\!\!}$ equipped with the topology induced 
from that of $\Rn{n}$ implies that $\. f \.$ is continuous 
according to the corollary given in~\cite[Chapitre~II, page~20]{BouEVT81}. 

\smallskip

\textasteriskcentered \ 
Finally, given any point ${a \in {\rb{C}} \;\!\!}$ 
and any straight line $L$ in $\Rn{n}$ passing through $a$, we have either $C \cap L = \Oset \.$, 
which obviously yields $\disp \lim_{x \goes a} \. f \. (x) = \:\! {+\infty}$ 
with $x \in C \cap L$, or $C \cap L \neq \Oset \.$, 
which first implies ${a \in {\rb{C \cap L}} \;\!\!}$ by Lemma~\ref{lem:rb-line}, 
and then $\disp \lim_{x \goes a} \. f \. (x) = \:\! {+\infty}$ with $x \in C \cap L$ 
by the Main~Theorem since ${{\Epi{\rest{f}{C \cap L \.}}} \.}$ is strictly convex. 

\smallskip

In both cases, we obtain $\disp \lim_{x \goes a} \. f \. (x) = \:\! {+\infty}$ 
by Lemma~\ref{lem:line-infinity}. 
\end{proof}





\bigskip
\bigskip
\bigskip


\bibliographystyle{acm}
\bibliography{math-biblio-convexity}

\end{document}